\documentclass[11pt,a4paper]{amsart}
\usepackage{amsfonts}
\usepackage{amsmath}
\usepackage{amssymb}
\usepackage{amsthm}
\usepackage{mathrsfs}
\usepackage{graphicx}

\newtheorem{theorem}{Theorem}[section]
\newtheorem{lemma}[theorem]{Lemma}

\newtheorem{remark}[theorem]{Remark}

\newtheorem{Conjecture}{Conjecture}

\newcommand{\dd}{\mathrm{d}}
\newcommand{\R}{{\mathord{\mathbb R}}}
\newcommand{\C}{{\mathord{\mathbb C}}}

\newcommand{\N}{{\mathord{\mathbb N}}}

\DeclareMathOperator{\sgn}{sgn}
\DeclareMathOperator{\dom}{Dom}
\DeclareMathOperator{\dist}{dist}
\usepackage[normalem]{ulem}
\usepackage{color}
\definecolor{DarkBlue}{rgb}{0,0.1,0.7}

\newcommand\soutD{\bgroup\markoverwith
{\textcolor{DarkBlue}{\rule[.5ex]{2pt}{1pt}}}\ULon}
\newcommand{\Hm}[1]{\leavevmode{\marginpar{\tiny%
$\hbox to 0mm{\hspace*{-0.5mm}$\leftarrow$\hss}%
\vcenter{\vrule depth 0.1mm height 0.1mm width \the\marginparwidth}%
\hbox to
0mm{\hss$\rightarrow$\hspace*{-0.5mm}}$\\\relax\raggedright #1}}}

\title[Nodal sets of thin curved layers]{Nodal sets of thin curved layers}
\author{David Krej\v{c}i\v{r}\'{i}k \ and \ Mat\v{e}j Tu\v{s}ek}
\date{11 June 2014}

\address{
Department of Theoretical Physics,
Nuclear Physics Institute,
Academy of Sciences,
250\,68 \v{R}e\v{z}, Czech Republic
}
\email{krejcirik@ujf.cas.cz}

\address{
Department of Mathematics,
Faculty of Nuclear Sciences and Physical Engineering,
Czech Technical University in Prague,
Trojanova 13, 120\,00 Prague 2, Czech Republic
}
\email{tusekmat@fjfi.cvut.cz}

\begin{document}
\begin{abstract}
This paper is concerned with the location of nodal sets 
of eigenfunctions of the Dirichlet Laplacian 
in thin tubular neighbourhoods of hypersurfaces 
of the Euclidean space of arbitrary dimension.
In the limit when the radius of the neighbourhood tends to zero,
it is known that spectral properties of the Laplacian are
approximated well by an effective Schr\"odinger operator 
on the hypersurface with a potential expressed solely
in terms of principal curvatures.
By applying techniques of elliptic partial differential equations,
we strengthen the known perturbation results  
to get a convergence of eigenfunctions in H\"older spaces.
This enables us in particular to conclude that
every nodal set has a non-empty intersection with the boundary of 
the tubular neighbourhood.
\end{abstract}
\maketitle

\section{Introduction}
The celebrated \emph{nodal-line conjecture} of L.~E.~Payne's 
from 1967~\cite[Conj.~5]{Payne1}
states that any membrane with fixed boundary vibrating at 
the second lowest resonant frequency cannot have 
a closed nodal curve.
Mathematically, representing the membrane
by a bounded domain (i.e.\ open connected set) $\Omega \subset \R^2$, 
the resonant frequencies~$\sqrt{\lambda}$ 
and corresponding vibrating modes~$u$  
are obtained as solutions of the boundary-value problem
\begin{equation}\label{bv}
\left\{
\begin{aligned}
  -\Delta u &= \lambda u
  && \mbox{in} \quad \Omega \,,
  \\
  u &= 0
  && \mbox{on} \quad \partial\Omega \,.
\end{aligned}
\right.
\end{equation}
It is customary to arrange the eigenvalues 
in a non-decreasing sequence (counting multiplicities) 
$
  0 < \lambda_1 \leq \lambda_2 \leq \lambda_3 \leq \dots   
$
and choose the corresponding eigenfunctions 
$u_1, u_2, u_3, \dots$ real-valued.
It is well known that~$\lambda_1$ is non-degenerate 
and that~$u_1$ does not change sign in~$\Omega$.
Since the eigenfunctions are mutually orthogonal in $L^2(\Omega)$,
the \emph{nodal set} 
\begin{equation}\label{nodal.set}
  \mathcal{N}(u_n) := u_n^{-1}(\{0\}) 
\end{equation}
is non-trivial whenever $n \geq 2$
and forms peculiar shapes (also known as Chladni's patterns):
various crossing curves or closed loops.
The connected components of $\Omega\setminus\mathcal{N}(u_n)$
are called \emph{nodal domains} of~$u_n$.
It turns out that the shape of the nodal set
is related to acoustic properties of the membrane.
The conjecture then states that $\mathcal{N}(u_2)$
cannot form a closed curve, or more generally,
in any dimension and for an arbitrary~$\Omega$,
that the nodal set touches the boundary:
\begin{Conjecture}\label{Conj1}
$\overline{\mathcal{N}(u_2)} \cap \partial\Omega \not= \emptyset$
for any open domain $\Omega \subset \R^d$, $d \geq 2$.
\end{Conjecture}

If~$\Omega$ is not sufficiently regular or it is unbounded,
\eqref{bv}~should be interpreted as a spectral problem
for the Dirichlet Laplacian $-\Delta_D^{\Omega}$ 
in the Hilbert space $L^2(\Omega)$.  
The conjecture still makes sense even for unbounded~$\Omega$,
provided that two eigenvalues below the essential spectrum exist.
We note that the spectral problem~\eqref{bv} 
for unbounded tubular domains arises 
in the context of stationary Schr\"odinger equation
in \emph{quantum waveguides}
(see \cite{DE,KKriz,Haag-Lampart-Teufel_2014} 
and references therein),
where the eigenfunctions have the physical interpretation 
of quantum bound states. 

Since the solutions of~\eqref{bv} are analytic in~$\Omega$, 
each $\mathcal{N}(u_n)$ decomposes into the disjoint union
of an analytic $(d-1)$-dimensional manifold
and a singular set contained in a countable number
of analytic $(d-2)$-dimensional manifolds
(cf.~\cite{Caffarelli-Friedman_1985}).
The \emph{Courant nodal domain theorem} then states that,
if the boundary~$\partial \Omega$ is sufficiently regular,
the $n$th~eigenfunction~$u_n$
has at most~$n$ nodal domains
(cf.~\cite{CH1,Chavel,Alessandrini_1998}).
In particular, $u_2$~has exactly two nodal domains.

The restriction to the second eigenfunction in Conjecture~\ref{Conj1} 
is essential because higher eigenfunctions are known 
to have nodal sets which possibly do not intersect the boundary.
As a matter of fact, there exist counterexamples even
to Conjecture~\ref{Conj1} as stated here,
given by domains~$\Omega$ 
which are multiply connected 
\cite{H2ON,Fournais,Kennedy_2011}
or unbounded \cite{FK2}. 
More significantly, 
it has been shown recently by Kennedy~\cite{Kennedy_2013} 
that the conjecture does not hold for $d \geq 3$, 
even if the domain~$\Omega$ is bounded and simply connected.
The current status is that the validity of Conjecture~\ref{Conj1} 
constitutes an open problem ``only'' for 
simply connected bounded domains in $d=2$.

Although the conjecture is known to be violated in general,
it is important to identify domains for which it is satisfied,
and moreover, provide information on the location
and geometry of the nodal set.
The most general result obtained so far
was given by Melas~\cite{Melas},
who showed that Conjecture~\ref{Conj1} holds in the case
of convex planar domains (cf.~also~\cite{Alessandrini_1994}).
Independently of Melas, Jerison proved Conjecture~\ref{Conj1}
in the case of sufficiently \emph{long and thin convex} 
planar domains~\cite{Jerison0} 
and later extended the proof to higher dimensions~\cite{Jerison}.
Furthermore, Jerison's method enabled him to locate 
the nodal set in the two-dimensional case
near the zero of an ordinary differential equation
associated to the convex domain in a natural way~\cite{Jerison1}
(cf.~also \cite{GJerison,Grieser-Jerison_2009}).

The only positive result for non-convex 
and possibly multiply connected domains in any dimension
is that of Freitas and one of the present authors
obtained for \emph{thin curved tubes} of arbitrary cross-sec\-tion~\cite{FK4}.
The result resembles that of Jerison's in that 
we also consider thin domains
and locate the nodal set near the zeros of an ordinary differential equation
associated to the curved tube, however, 
the technical approach is actually very different.
Instead of a usage of a trial function
and refined applications of the maximum principle 
developed by Jerison, we rather rely on a singular
perturbation theory in a Hilbert-space setting
(although the maximum principle, in addition to other techniques,
is also used eventually).
Moreover, our method works in all dimensions 
and also for higher eigenfunctions. 

The objective of the present paper is to prove Conjecture~\ref{Conj1}
for \emph{thin curved layers}~$\Omega_\varepsilon$, 
i.e.\ $\varepsilon$-tubular neighbourhoods 
of hypersurfaces~$\Sigma$ in~$\R^d$, 
with $d \geq 2$ and $\varepsilon \to 0$.
The difference with respect to thin tubes considered in~\cite{FK4},
where the underlying manifold was just a curve 
(of dimension~$1$ and codimension~$d-1$),
is that now the asymptotic geometry is that of a manifold
of dimension $d-1$ (and codimension~$1$).
It leads both to conceptual and technical complications.
Indeed, the location of the nodal sets is presently determined
by a \emph{partial} differential equation associated to the hypersurface,
which itself could have closed nodal sets
(even if $n=2$, cf.~\cite{Freitas1}).
More importantly, the method of~\cite{FK4} 
to pass from $L^2$- to $C^0$-estimates 
needed for the location of the nodal sets
does not work in this more complex situation.
Therefore new ideas that we explain below had to be developed.
Finally, contrary to~\cite{FK4} we do not exclude from the present paper
the possibility that the domain~$\Omega_\varepsilon$ is unbounded,
so that some of the results are new even in $d=2$, 
when the geometrical setting coincides with that of~\cite{FK4}.

Let us state a consequence of the results established 
in this paper.
\begin{theorem}\label{Thm.Conj}
Let~$\Omega_\varepsilon$ be an $\varepsilon$-tubular neighbourhood 
of a smooth hypersurface~$\Sigma$ in~$\R^d$, $d \geq 2$. 
Assume hypotheses \ref{eq:eps_bound}, \ref{eq:res_conv_ass},
and \ref{eq:reg_ass_final} stated below.
If~$\Sigma$ is unbounded, let us also assume that 
the Dirichlet Laplacian $-\Delta_D^{\Omega_\varepsilon}$ 
possesses $m\geq 2$ 
eigenvalues below the essential spectrum
for all sufficiently small~$\varepsilon$.
Then for any $N\in\{2,3,\ldots\}$ or $N\in\{2,3,\ldots,m\}$ 
in the bounded or unbounded case, respectively, there exists a positive constant
$\varepsilon_0$ 
depending on the geometry of~$\Sigma$ 
such that 
$$
  \overline{\mathcal{N}(u_n)} \cap \partial\Omega_\varepsilon \not= \emptyset
$$
holds for all $\varepsilon \leq \varepsilon_0$
and every $n \in \{2,\dots,N\}$.
In particular, 
Conjecture~\ref{Conj1} holds for~$\Omega_\varepsilon$ 
whenever $\varepsilon \leq \varepsilon_0$. 
\end{theorem}

As a matter of fact, Theorem~\ref{Thm.Conj} follows as a consequence
of a stronger result, Theorem~\ref{theo:sign_change}, 
in which we establish a \emph{local convergence}
of nodal sets for all eigenfunctions
of $-\Delta_D^{\Omega_\varepsilon}$
to nodal sets of eigenfunctions of a sum of
a $(d-1)$-dimensional Schr\"odinger operator~$h_\mathrm{eff}$ 
in~$L^2(\Sigma)$, cf.~\eqref{h.eff}, 
and the one-dimensional Dirichlet Laplacian in the cross-section.

The structure of the paper is as follows.
In the next Section~\ref{Sec.pre} 
we collect basic facts and assumptions
about the geometry of the hypersurface~$\Sigma$ 
and its tubular neighbourhood~$\Omega_\varepsilon$.  
We also introduce the effective 
Schr\"odinger operator~$h_\mathrm{eff}$ in~$L^2(\Sigma)$,
which governs the location of nodal sets in~$\Omega_\varepsilon$
for small~$\varepsilon$, and recall from~\cite{KRT}
a result on a \emph{norm-resolvent convergence} 
of $-\Delta_D^{\Omega_\varepsilon}$ to~$h_\mathrm{eff}$.
After an appropriate identification of Hilbert spaces,
the operator convergence implies the convergence 
of the respective eigenfunctions in $L^2(\Omega_\varepsilon)$,
however, this convergence is not sufficient to deduce 
the convergence of nodal sets.

Section~\ref{Sec.conv} consists of a number of subsections
devoted to a passage from the $L^2$-convergence 
to a \emph{uniform}, $C^0$-convergence. 
The main idea is to adapt techniques known 
from elliptic regularity theory 
for partial differential equations
in order to establish a convergence of eigenfunctions
in higher-order Sobolev spaces and apply the Sobolev embedding theorem. 
We believe that our approach is of independent interest
because we deal with a highly singular problem
when some coefficients of the partial differential equation
depends on~$\varepsilon^{-2}$.
In the last subsection, we additionally use the maximum principle
to deduce a $C^{0,1}$-convergence needed for the convergence
of nodal sets.

The local convergence of nodal sets 
is deduced from the previous convergence results 
in Section~\ref{Sec.nodal}. The main result of this paper 
is given by Theorem~\ref{theo:sign_change} there.
The last Section~\ref{Sec.special} deals with a special class
of underlying manifolds $\Sigma$ for which 
one can extend the local results of 
Section~\ref{Sec.nodal}
to a \emph{global convergence} along the whole nodal set. 
For this class of manifolds,
we also demonstrate the convergence (in the Hausdorff sense) of nodal domains
employing the Courant nodal domain theorem.

When finishing this paper a related work of Lampart~\cite{Lampart} appeared.
On the one hand, it is concerned with 
the much more general geometrical setting of the Laplacian
on a diminishing fibre bundle of arbitrary dimension. 
On the other hand, the base space of the fibre
is assumed to be compact and of dimension at most~$3$ in~\cite{Lampart},
which would correspond to the restriction to compact $\Sigma$ and $d\leq 4$
in our setting. It is remarkable that we do not need these constraints
to prove Conjecture~\ref{Conj1}, but our fibre is always one-dimensional 
(although there seem not to be any technical limitations 
with extending our methods to higher-dimensional cross-sections, 
cf.~\cite{FK4}).
The present work and~\cite{Lampart} are also independent 
from the point of view of the technical handling of the limit $\varepsilon \to 0$.
While Lampart starts with an adiabatic perturbation theory 
developed for fibre bundles in \cite{Lampart-Teufel}
and finally ends up with $\mathcal{H}^1$-estimates,
we rather rely on the norm-resolvent convergence established in~\cite{KRT}
and proceed to higher-order Sobolev spaces.

\section{Notation and preliminary results}\label{Sec.pre}
At the beginning let us point out that we will use both the ``prefix'' and
``postfix'' notation for partial derivatives. In the prefix notation a partial
derivative acts on the whole expression to the right, whereas in the postfix
notation a partial derivative acts only on the closest term to the left. This
convention shall save us from the over-bracketing.

\subsection{Geometry of the layer}
Mostly we adopt the notation of \cite{KRT}. In particular let $\Sigma$ be a
connected orientable $C^l$-hypersurface with $l\geq 4$  in $\R^d$ ($d\geq 2$)
with the Riemannian metric $g$ induced by the embedding. It may or may not have
a boundary, if it does, then we assume that the boundary is also $C^l$-smooth.
The orientation is given by a globally defined unit normal field $n:\Sigma\to
S^{d-1}$. For any $p\in\Sigma$, we introduce the Weingarten map
$$ L: T_p\Sigma \to T_p\Sigma: \  \big\{\xi \mapsto -\dd n(\xi)\big\}.$$
If $x^1,\ldots, x^{d-1}$ are some local coordinates on $\Sigma$, 
then $L$ in the coordinate frame $(\partial_1,\ldots, \partial_{d-1})$ (where we
abbreviate $\partial_\mu=\partial_{x^{\mu}}$) has the matrix representation
$(L^{\mu}_{\ \nu})$ that reads $L^{\mu}_{\ \nu} = g^{\mu\rho} h_{\rho\nu}$,
where $(g^{\mu\rho})$ and $(h_{\rho\nu})$ is the matrix representation of
$g^{-1}$ and the second fundamental form of $\Sigma$, respectively. We set
$|g|:=\det{g}$.

The eigenvalues of $L$ are the principal curvatures,
$\kappa_1,\ldots,\kappa_{d-1}$, of $\Sigma$. 
Re\-mark that they are in general given only locally. However, there are $d-1$
invariants of $L$ given by the formula
$$
  K_{\mu}:=\binom{d-1}{\mu}^{-1}
  \sum_{\alpha_{1}<\ldots< \alpha_{\mu}}
  \kappa_{\alpha_{1}}\ldots\kappa_{\alpha_{\mu}},\qquad\mu=1,\ldots,d-1,
$$
that are globally defined $C^{l-2}$-smooth functions. 

Given $I:=(-1,1)$ and $\varepsilon>0$, we define a layer $\Omega_\varepsilon$ of
width $2\varepsilon$ along $\Sigma$ as the interior of the image of the mapping
\begin{equation}\label{eq:layer_def}
  \mathscr{L}: \ 
  \Sigma\times I\to\R^d: \ \big\{(x,u)\mapsto x+\varepsilon u n\big\},
\end{equation}
i.e., 
$$
  \Omega_\varepsilon
  :=\mathrm{int}\,\mathscr{L}(\Sigma\times I).
$$
We always assume that $\Omega_\varepsilon$ does not overlap itself, i.e.,
\begin{equation}\label{eq:eps_bound} \tag*{$\langle A1\rangle$}
  \varepsilon<\varrho_{m}
  :=\big(\max\big\{\|\kappa_{1}\|_{\infty},
  \ldots,\|\kappa_{d-1}\|_{\infty}\big\}\big)^{-1}
   \ \& \quad
  \mathscr{L} \mbox{ is injective},
\end{equation}
which implicitly involves the assumption 
that the principal curvatures are bounded.

In the sequel we set $x^d:=u$, $\partial_{d}=\partial_{u}$, and
$x_{\shortparallel}=(x^1,\ldots,x^{d-1})$. 
Also, unless otherwise stated, 
we will always assume the range of Greek and Latin indices to be $1,\ldots,d-1$
and $1,\ldots,d$, respectively. 

The metric induced by \eqref{eq:layer_def} is of the following block form, when
written in the coordinate frame $(\partial_1,\ldots, \partial_d)$, 
\begin{equation}\label{eq:layer_metric}
  (G_{ij})=
  \begin{pmatrix}
         (G_{\mu\nu}) & 0\\0 & \varepsilon^2
  \end{pmatrix},
  \qquad 
  G_{\mu\nu}=g_{\mu\rho}
  (\delta^{\rho}_{\sigma} -\varepsilon u L^{\rho}_{\ \sigma})
  (\delta^{\sigma}_{\nu} -\varepsilon u L^{\sigma}_{\ \nu}).
\end{equation}

In a straightforward manner one can derive the following matrix estimate
\begin{equation}\label{eq:metric_est}
 C_{-}g_{\mu\nu}\leq G_{\mu\nu}\leq C_{+}g_{\mu\nu}.
\end{equation}
with $C_{\pm}:=(1\pm\varepsilon\varrho_{m}^{-1})^2=1+\mathcal{O}(\varepsilon)$.

\subsection{Dirichlet Laplacian}
The Dirichlet Laplacian, $-\Delta_{D}^{\Omega_{\varepsilon}}$, on
$\Omega_{\varepsilon}$ is the self-adjoint operator 
in $L^2(\Omega_{\varepsilon})$
associated with the quadratic form
$$
  Q[\psi]:=\|\nabla\psi\|_{L^2(\Omega_{\varepsilon})}^2
  \,, \qquad
  \dom{Q}:=\mathcal{H}_{0}^{1}(\Omega_{\varepsilon})
  \,.
$$
It was demonstrated in \cite{KRT} that $-\Delta_{D}^{\Omega_{\varepsilon}}$ is
unitarily equivalent to 
\begin{equation*}
  H :=U(-\Delta_{D}^{\Omega_{\varepsilon}})U^{-1} 
  =
-|g|^{-1/2}\partial_{\mu}|g|^{1/2}G^{\mu\nu}\partial_{\nu}-\varepsilon^{-2}
\partial_{u}^{2}+V,
\end{equation*}
acting in
$\mathscr{H}:=L^{2}(\Sigma\times I,|g|^{1/2}\dd x^1\ldots\dd x^d)$ 
with $\dom{H^{1/2}}=U[\dom{Q}]$, where
$U:\psi\mapsto\mathrm{e}^J (\psi\circ\mathscr{L})$
with
\begin{align*}
 & J := \frac{1}{4} \ln\frac{|G|}{|g|} 
  = \frac{1}{2}\ln\left[1+\sum_{\mu=1}^{d-1}(-\varepsilon u)^{\mu} 
  \binom{d-1}{\mu}K_{\mu}\right] , \\
 & V := |g|^{-1/2} \partial_{i}|g|^{1/2} G^{ij} \partial_{j} J 
 + J_{,i} G^{ij} J_{,j}\, .
\end{align*}

\begin{remark}\label{rem:metric_cont}
Under our regularity assumption on $\Sigma$, the entries of $(g^{\mu\nu})$ are
$C^{l-1}$-functions, the entries of $(L^{\mu}_{\ \nu})$ and $(G^{\mu\nu})$ are
$C^{l-2}$-functions, 
and $V$ is only a $C^{l-4}$ function in any local
coordinates. In some more detail, the entries of $(G^{\mu\nu})$ are polynomials
in the parameter $\varepsilon$ with $C^{l-2}$-smooth coefficients and
$\lim_{\varepsilon\to 0}G^{\mu\nu}=g^{\mu\nu}$. This may be easily observed,
e.g. from the Cayley-Hamilton theorem that states
 $$\prod_{i=0}^{d-1}(\kappa_i-L)=0,$$
 and so
 \begin{equation}\label{eq:Cal_Ham}
 0=\prod_{i=0}^{d-1}(\mathit{Id}-\varepsilon u L+(\varepsilon u\kappa_i
-1)\mathit{Id})=\sum_{i=0}^{d-1}A_i(\mathit{Id}-\varepsilon u L)^i,
 \end{equation}
 where $A_{i}$ is a product of a polynomial in $\{\varepsilon u\kappa_{j},\
j=1,\ldots,d-1\}$ and the identity map, $\mathit{Id}$. In particular,
 $A_{d-1}=\mathit{Id}$ and $A_0=\prod_{i=1}^{d-1}(\varepsilon u\kappa_i-1)
\mathit{Id}$. Under the assumption \ref{eq:eps_bound}, $A_0\neq 0$. Consequently
we may compute 
 $(\mathit{Id}-\varepsilon u L)^{-1}$ 
from \eqref{eq:Cal_Ham}
as a finite sum and then substitute the result to the formula for
$(G^{\mu\nu})$, which is easy to derive from \eqref{eq:layer_metric}.
\end{remark}

\subsection{Effective Hamiltonian}
Let $(E_m,\chi_m),\ m\in\N \equiv \{1,2,\dots\}$, 
be the $m$th eigenpair 
of $-\partial_{u}^{2}$ in $L^2(I)$, 
subject to the Dirichlet boundary condition on $\partial I$. 
We have $E_m=(m\pi/2)^2$ and, in particular, we choose
\begin{equation}\label{chi1}
  \chi_{1}(u) := \cos(\pi u/2).
\end{equation}
For small values of $\varepsilon$, $H$ behaves like
\begin{equation*}
  H_{0}:= h_{\mathrm{eff}}-\varepsilon^{-2}\partial_{u}^{2} 
  \ \simeq \ 
  h_{\mathrm{eff}} \otimes 1  + 1 \otimes (-\varepsilon^{-2}\partial_{u}^{2})
\end{equation*}
acting on $\mathscr{H}=L^2(\Sigma,|g|^{1/2}\dd x^1\ldots\dd x^{d-1})\otimes
L^2(I,\dd u)$ with the form domain $\dom{H_0^{1/2}}:= \dom{H^{1/2}}$.
Here the so-called effective Hamiltonian $h_{\mathrm{eff}}$ reads
\begin{equation}\label{h.eff}
 h_{\mathrm{eff}}=
-\Delta_{g}+V_{\mathrm{eff}},\quad\dom{h_{\mathrm{eff}}}=\dom{-\Delta_g},
\end{equation}
where
$-\Delta_{g} := -|g|^{-1/2}\partial_{\mu}|g|^{1/2}g^{\mu\nu}\partial_{\nu}$ 
is the Laplace-Beltrami operator on $\Sigma$ subject to 
Dirichlet boundary conditions and 
\begin{equation*}
 V_{\mathrm{eff}}:=-\frac{1}{2}\sum_{\mu=1}^{d-1}\kappa_{\mu}^{2}+\frac{1}{4}
\left(\sum_{\mu=1}^{d-1}\kappa_{\mu}\right)^2=\binom{d-1}{2}K_2-\left(\frac{d-1}
{2}\,K_1\right)^2\, .
\end{equation*}
To be more precise, we reproduce here Theorems 6.3 and 7.1 of \cite{KRT} in a
slightly modified form.
\begin{theorem}[\cite{KRT}]\label{Thm.KRT} 
Let $H_{\mathrm{ren}}:=H-\varepsilon^{-2} E_1$,
$H_{0,\mathrm{ren}}:=H_0-\varepsilon^{-2} E_1$.
Assume \ref{eq:eps_bound} and
\begin{equation}\tag*{$\langle A2\rangle$}\label{eq:res_conv_ass}
|\nabla_{\!g}\kappa_{\mu}|_{g}
  \,, \
  \Delta_{g}\kappa_{\mu}  \in L^{\infty}(\Sigma),
\end{equation}
where $|\nabla_{\!g}f|_{g} := \sqrt{ f_{,\mu} g^{\mu\nu} f_{,\nu} }$.
Then  for any $k\in\mathrm{Res}\,H_{0,\mathrm{ren}}$, 
positive constants~$\varepsilon_0$ and~$C$ exist such that, 
for all $\varepsilon<\varepsilon_0$,
$k\in\mathrm{Res}\,H_{\mathrm{ren}}$ and
\begin{equation}\label{eq:res_conv}
   \|(H_{\mathrm{ren}}+k)^{-1}-(H_{0,\mathrm{ren}}+k)^{-1}\|\leq C\varepsilon .
\end{equation}
Similarly,  for any $k\in\mathrm{Res}\,h_{\mathrm{eff}}$, positive constants
$\tilde\varepsilon_0$ and $\tilde C$ exist such that, for all
$\varepsilon<\tilde\varepsilon_0$,
$k\in\mathrm{Res}\,H_{\mathrm{ren}}$ and
\begin{equation*}
   \|(H_{\mathrm{ren}}+k)^{-1}-(h_{\mathrm{eff}}+k)^{-1}\oplus 0\|\leq \tilde
C\varepsilon .
\end{equation*}
\end{theorem}
\noindent Here  $(h_{\mathrm{eff}}+k)^{-1}\oplus 0$ acts on
$\mathscr{H}_{1}\oplus\mathscr{H}_{1}^{\bot}$,
where $\mathscr{H}_{1}$ is spanned by the functions of the form 
$f\otimes\chi_1$ 
with $f\in L^2(\Sigma,|g|^{1/2}\dd x^1\ldots\dd x^{d-1})$.

\section{Convergence of eigenfunctions}\label{Sec.conv}
\subsection{$L^2$-convergence}
Denote the eigenvalues of $H_{0,\mathrm{ren}}$ 
in a non-decreasing order (counting
multiplicity) by $\{\lambda_{n}^{0}\}_{n=1}^\infty$. 
Let $\mu=\lambda_{k}^{0}=\ldots =\lambda_{k+p-1}^{0}$ be a $p$-degenerated
eigenvalue of $H_{0,\mathrm{ren}}$, 
$k,p \geq 1$. 
Due to \eqref{eq:res_conv}, one can prove (see again \cite{KRT} for details)
that for all sufficiently small $\varepsilon$, the number of eigenvalues of
$H_{\mathrm{ren}}$ in the contour $\Gamma$ of radius smaller than the isolation
distance of $\mu$  and centred at $\mu$ equals $p$. Moreover, there exists a
family of linearly independent eigenfunctions $\psi_{n}$, 
with $n=k,\dots,k+p-1$, of~$H_{\mathrm{ren}}$ associated with
these eigenvalues of~$H_{\mathrm{ren}}$,
which we denote by 
$\lambda_k, \dots, \lambda_{k+p-1}$
(they are not necessarily sorted in a non-decreasing order,
but they are counted according to multiplicities), and a base
$(\psi_k^0,\ldots,\psi_{k+p-1}^0)$
of the eigenspace associated with the eigenvalue $\mu$ of~$H_{0,\mathrm{ren}}$
such that the following result holds (see also \cite[Corol.~8.1]{KRT}).

\begin{theorem}
With the notation introduced above
and under the hypotheses of Theorem~\ref{Thm.KRT}, 
\begin{equation*}
 |\lambda_n-\lambda_{n}^{0}|=\mathcal{O}(\varepsilon),\quad
\|\psi_n-\psi_{n}^{0}\|=\mathcal{O}(\varepsilon),
\end{equation*}
as $\varepsilon\to 0$.
\end{theorem}

\begin{remark}[Eigenpairs of $H_{0,\mathrm{ren}}$] \label{rem:H_0_spec} It is
obvious that the eigenpairs $(\lambda_{n}^{0},\psi_{n}^{0})$ are of the form
$(\sigma_i+\varepsilon^{-2}(E_j-E_1),\phi_i\otimes\chi_j)$, where
$(\sigma_i,\phi_i)$ is an eigenpair of $h_{\mathrm{eff}}$. Thus for any $N\in\N$
there is $\varepsilon(N)$ such that for all $\varepsilon<\varepsilon(N)$ and
$n=1,\ldots,N$, we have $\lambda_{n}^{0}=\sigma_n$ and
$\psi_{n}^{0}=\phi_{n}\otimes\chi_1$.
\end{remark}

\subsection{$\mathcal{H}^{1}$-convergence}
Subtracting eigenvalue equations for the operators
$H_{\mathrm{ren}}$ and $H_{0,\mathrm{ren}}$, 
we arrive at the following equation for $\psi:=\psi_{n}-\psi_{n}^{0}$
\begin{equation}\label{eq:psi_diff}
\begin{split}
 -|g|^{-1/2}\partial_{\mu}|g|^{1/2}G^{\mu\nu}\partial_{\nu}\psi&+\varepsilon^{-2
}(-\partial_{u}^{2}-E_{1})\psi=|g|^{-1/2}\partial_{\mu}|g|^{1/2}a^{\mu\nu}
\partial_{\nu}\psi_{n}^{0}\\
 &+(\lambda_{n}-V)\psi+(V_{\mathrm{eff}}-V+\lambda_{n}-\lambda_{n}^{0})\psi_{n}^
{0},
\end{split}
\end{equation}
where $a^{\mu\nu}:=G^{\mu\nu}-g^{\mu\nu}$. Next we multiply (using the scalar
product of $\mathscr{H}$) the both sides of this result by $\psi$ to obtain
\begin{equation*}
\begin{split}
\langle\partial_{\mu}\psi,G^{\mu\nu}\partial_{\nu}\psi\rangle+&\varepsilon^{-2}
(\|\partial_{u}\psi\|^2-E_{1}\|\psi\|^2)=-\langle\partial_{\mu}\psi_{n}^{0},a^{
\mu\nu}\partial_{\nu}\psi\rangle\\&+\langle(\lambda_{n}-V)\psi,
\psi\rangle+\langle(V_{\mathrm{eff}}-V+\lambda_{n}-\lambda_{n}^{0})\psi_{n}^{0},
\psi\rangle.
\end{split}
\end{equation*}
By the Cauchy-Schwarz and Young inequalities we have
\begin{equation}\label{eq:product_est}
\begin{split}
\langle\partial_{\mu}\psi,&G^{\mu\nu}\partial_{\nu}\psi\rangle+\varepsilon^{-2}
(\|\partial_{u}\psi\|^2-E_{1}\|\psi\|^2)
\leq \frac{\delta}{2}\|\nabla_{\! a}\psi_{n}^{0}\|^2\\
&+\frac{1}{2\delta}\|\nabla_{\! a}\psi\|^2+\|\lambda_{n}-V\|_{\infty}\|\psi\|^2
+\|V_{\mathrm{eff}}-V+\lambda_{n}-\lambda_{n}^{0}\|_{\infty}\,\|\psi_{n}^{0}\|\,
\|\psi\|.
\end{split}
\end{equation}
with any $\delta>0$.

Estimates~\eqref{eq:metric_est} imply
$$-(1-C_{+}^{-1})g^{\mu\nu}\leq  a^{\mu\nu}\leq(C_{-}^{-1}-1)g^{\mu\nu}.$$
For $\delta\leq 1/2$, we have
$$G^{\mu\nu}-\frac{1}{2\delta}a^{\mu\nu}\geq
g^{\mu\nu}+\left(1-\frac{1}{2\delta}\right)(C_{-}^{-1}-1)g^{\mu\nu}.$$
If $\varepsilon\varrho_{m}^{-1}\leq 1-\sqrt{2/3}$, then $(C_{-}^{-1}-1)\leq
1/2$.  Consequently we may set $\delta=(C_{-}^{-1}-1)$ to conclude that
$$G^{\mu\nu}-\frac{1}{2\delta}a^{\mu\nu}\geq
\frac{1}{2}g^{\mu\nu}+(C_{-}^{-1}-1)g^{\mu\nu}\geq\frac{1}{2}g^{\mu\nu}.$$
Now, from \eqref{eq:product_est} it follows that
\begin{equation*}
\begin{split}
 \frac{1}{2}\|\nabla_{\!
g}\psi\|^2+\varepsilon^{-2}(\|\partial_{u}\psi\|^2-E_{1}\|\psi\|^2)&\leq\frac{C_
{-}^{-1}-1}{2}\|\nabla_{\! a}\psi_{n}^{0}\|^2+\mathcal{O}(\varepsilon^2)\\
 &\leq\frac{(C_{-}^{-1}-1)^2}{2}\|\nabla_{\!
g}\psi_{n}^{0}\|^2+\mathcal{O}(\varepsilon^2)
 \\
 &=\mathcal{O}(\varepsilon^2).
\end{split}
\end{equation*}
Here we used that 
$\|\psi_{n}^{0}\|<\infty,\, \|\nabla_{\! g}\psi_{n}^{0}\|<\infty$, 
and
\begin{equation}\label{eq:pot_bound}
\begin{aligned}
\|\lambda_{n}-V\|_{\infty}&=\mathcal{O}(1),
&&& (C_{-}^{-1}-1)&=\mathcal{O}(\varepsilon),
\\
\|V_{\mathrm{eff}}-V+\lambda_{n}-\lambda_{n}^{0}\|_{\infty}
&=\mathcal{O}(\varepsilon),
&&& \|\psi\| &=\mathcal{O}(\varepsilon),
\end{aligned}
\end{equation}
as $\varepsilon\to 0$. Since for all $\psi\in\mathcal{H}_{0}^{1}(\Sigma\times
I)$, $\|\partial_{u}\psi\|^2\geq E_{1}\|\psi\|^2$ 
by Fubini's theorem and the Poincar\'{e} inequality
\begin{equation}\label{Poincare}
  \forall \varphi \in \mathcal{H}_{0}^{1}(I) \,, \qquad
  \int_I |\partial_{u}\varphi|^2\, \dd u 
  \geq E_1 \int_I |\varphi|^2 \, \dd u \,,
\end{equation}
we finally obtain 
\begin{equation}\label{eq:H_1_convergence}
 \|\nabla_{\! g}\psi\|=\mathcal{O}(\varepsilon),\quad
\|\partial_{u}\psi\|=\mathcal{O}(\varepsilon).
\end{equation}

\begin{remark}
Let  $\mathcal{K}$ be a compact set in the range of a chosen local coordinate
map on $\Sigma$. Recall that the entries of $(g_{\mu\nu})$ are continuous. 
Then due to the 
local boundedness and positive definiteness
of $g$ we have
$$\inf_{\mathcal{K}}\|g\|^{-1}>0,\quad \inf_{\mathcal{K}}|g|>0,$$
where $\|g\|:=\|g\|_{\C^{d-1}\to\C^{d-1}}$.
Moreover, the matrix estimate $g_{\mu\nu}\leq\|g\|\,\delta_{\mu\nu}$ implies
$$\|g\|^{-1}\,\delta^{\mu\nu}\leq g^{\mu\nu}.$$
Since
$$
\begin{aligned}
  \langle\partial_{\mu}\psi,g^{\mu\nu}\partial_{\nu}\psi\rangle
  &\geq\langle\|g\|^ {-1}\,\delta^{\mu\nu}\partial_{\mu}\psi,
  \partial_{\nu}\psi\rangle
  \\
  &\geq \inf_{\mathcal{K}}\|g\|^{-1}\,
  \inf_{\mathcal{K}}|g|^{1/2}\int_{\mathcal{K}\times I}
  \sum_{i=1}^{d-1}|\partial_{i}\psi|^2\dd x,
\end{aligned}
$$
we conclude that
$$\int_{\mathcal{K}\times
I}\sum_{i=1}^{d-1}|\partial_{i}\psi|^2=\mathcal{O}(\varepsilon^2),$$
due to \eqref{eq:H_1_convergence}.
\end{remark}

\subsection{Interior $\mathcal{H}^2$-convergence}\label{sec:H2_conv}
In this section we assume that $\psi_{n}$ and $\psi_{n}^{0}$ are sufficiently
regular so that all their derivatives up to the $3$rd order exist in the
classical sense and are interchangeable. This holds true if
\begin{equation}\tag*{$\langle A3\rangle$}\label{eq:as3}
 \Sigma\in C^l,\ \partial\Sigma\in C^l,\quad\text{with }l>d/2+6,
\end{equation}
as one may see from Remark \ref{rem:metric_cont}, the elliptic regularity, and
the Sobolev embedding theorem. Our proof of the $\mathcal{H}^2$-convergence of
the eigenfunctions is motivated by the proof of the interior regularity for
elliptic operators (see \cite[Theo.~1 in $\S$6.3.1]{evans}). The major
difference is that we have to treat $\varepsilon^{-2}$-terms in  a special
manner.
In fact, due to the product character of the domain, we will prove the
regularity up to the transverse part of  the boundary, $\Sigma\times\partial I$.

Let $W$ be an open bounded subset of $\R^{d-1}$ such that the closure of $W$
(denoted by $\overline{W}$) 
lies in the range of a local coordinate map on $\Sigma$. Furthermore, let $U$
stand for an open precompact subset containing $\overline{W}$ and contained in
the range of the same coordinate map. Then there exists a smooth function
$\zeta$ (known as the bump function for $\overline{W}$ supported in $U$) with
the following properties: $0\leq\zeta\leq1$, 
$\zeta = 1$ on $\overline{W}$, 
and $\mathrm{supp}~\zeta\subset U$. For $k=1,2,\ldots,d$, set 
\begin{equation}\label{eq:test_funct}
v:=-\partial_{k}\zeta^2\partial_{k}\psi\quad(\text{no summation over }k).
\end{equation}
Using integration by parts carefully, we obtain
\begin{align}
  \langle -\partial_{\mu}G^{\mu\nu}\partial_{\nu}\psi,v\rangle_{U\times I} 
&=\langle\partial_{k}G^{\mu\nu}\partial_{\nu}\psi,\partial_{\mu}\zeta^2\partial_
{k}\psi\rangle_{U\times I} \label{eq:per_partes}
  \\
 \langle \psi,v\rangle_{U\times I}
  &=\langle\partial_{k}\psi,\zeta^2\partial_{k}\psi\rangle_{U\times I}
\label{eq:per_partes2}
  \\
  \langle -\partial_{u}^{2}\psi,v\rangle_{U\times I}
&=\langle\partial_{u}\partial_{k}\psi,\zeta^2\partial_{u}\partial_{k}
\psi\rangle_{U\times I},\label{eq:per_partes3}
\end{align}
where $\langle\cdot,\cdot\rangle_{U\times I}$ is the scalar product of
$L^{2}(U\times I,\dd x_1\ldots\dd x_d)$. The boundary terms vanish because
$\psi=\partial_{\mu}\psi=\partial_{\mu}\partial_{\nu}\psi=0$ on $U\times\partial
I$ and $\zeta=\partial_{j}\zeta^2\partial_{k}\psi=0$ on $\partial U\times I$.

If we  multiply (on $L^{2}(U\times I,\dd x_1\ldots\dd x_d)$) the both sides of
($\ref{eq:psi_diff}$) by $v$ we have
$$A_{1}+\varepsilon^{-2}A_{2}=B,$$
where
\begin{align*}
 A_{1}&:=\langle\partial_{k}G^{\mu\nu}\partial_{\nu}\psi,
 \partial_{\mu}\zeta^2\partial_{k}\psi\rangle_{U\times I},\nonumber\\
 A_{2}&:=\langle\partial_{u}\partial_{k}\psi,\zeta^2\partial_{u}\partial_{k}
\psi\rangle_{U\times I}-E_{1}\langle\partial_{k}\psi,
\zeta^2\partial_{k}\psi\rangle_{U\times I},\label{eq:poincare}\\
 B&:=\langle-(|g|^{-1/2})_{,\mu}G^{\mu\nu}|g|^{1/2}\partial_{\nu}\psi+f_{
\varepsilon}+(\lambda_{n}-V)\psi,v\rangle_{U\times I},\nonumber
\end{align*}
with
\begin{equation*}
 f_{\varepsilon}:=|g|^{-1/2}\partial_{\mu}|g|^{1/2}a^{\mu\nu}\partial_{\nu}\psi_
{n}^{0}+(V_{\mathrm{eff}}-V+\lambda_{n}-\lambda_{n}^{0})\psi_{n}^{0}.
\end{equation*}
Note that, as $\partial_{u}\zeta=0$ everywhere and $\partial_{\mu}\psi=0$ on
$U\times\partial I$, $A_{2}\geq 0$ if $k=1,\ldots,d-1$, due to Fubini's theorem
and the Poincar\'{e} inequality~\eqref{Poincare}.

\begin{remark}
In virtue of Remark~\ref{rem:metric_cont},
$$\sup_{U\times I}|a^{\mu\nu}|=\mathcal{O}(\varepsilon), \qquad \sup_{U\times
I}|\partial_\mu a^{\mu\nu}|=\mathcal{O}(\varepsilon).$$
This together with \eqref{eq:pot_bound} implies that
$$\int_{U\times I}|f_{\varepsilon}|^2=\mathcal{O}(\varepsilon^2).$$
\end{remark}

One immediately has
\begin{equation*}
 A_{1}=A_{1}'+A_{1}'',
\end{equation*}
where
\begin{equation*}
\begin{split}
 &A_{1}':=\langle
G^{\mu\nu}\partial_{\nu}\partial_{k}\psi,\zeta^2\partial_{\mu}\partial_{k}
\psi\rangle_{U\times I}\\
 &\begin{split}A_{1}'':=&\langle
G^{\mu\nu}_{,k}\partial_{\nu}\psi,2\zeta\zeta_{,\mu}\partial_{k}\psi\rangle_{
U\times I}+\langle
G^{\mu\nu}_{,k}\partial_{\nu}\psi,\zeta^2\partial_{\mu}\partial_{k}\psi\rangle_{
U\times I}\\
  &+\langle G^{\mu\nu}\partial_{k}\partial_{\nu}\psi,2\zeta\zeta_{,\mu}
\partial_{k}\psi\rangle_{U\times I}.
 \end{split}
\end{split}
\end{equation*}
Each part of this decomposition may be estimated separately. We have
\begin{equation}\label{eq:A_1'_est}
 A_{1}'\geq\gamma\langle\zeta^2\delta^{\mu\nu}\partial_{\nu}\partial_{k}\psi,
\partial_{\mu}\partial_{k}\psi\rangle_{U\times I}=\gamma\int_{U\times I}\zeta^2
|\tilde{D}\partial_{k}\psi|^2,
\end{equation}
where $\tilde{D}\psi:=(\partial_1 \psi,\ldots,\partial_{d-1}\psi)$ and the
positive constant $\gamma$ may be chosen as a uniform bound for all
$\varepsilon\in(0,\varepsilon_{U})$ with some fixed
$\varepsilon_{U}<\varrho_{m}$. In view of Remark~\ref{rem:metric_cont}, it is
also clear that the quantities
$$\sup_{U\times I}|G^{\mu\nu}| \qquad \text{and} \qquad \sup_{U\times
I}|G^{\mu\nu}_{,k}|$$
are uniformly bounded on $(0,\varepsilon_{U})$. Therefore, for
$\varepsilon<\varepsilon_{U}$, we can estimate 
\begin{equation}\label{eq:A_1''_est}
 \begin{split}
 |A_{1}''|&\leq C\int_{U\times
I}\left(\zeta|\tilde{D}\psi||\partial_{k}\psi|+\zeta|\tilde{D}\psi||\tilde{D}
\partial_{k}\psi|+\zeta|\partial_{k}\tilde{D}\psi||\partial_{k}\psi|\right)\\
 &\leq 2C\int_{U\times
I}\Big(\delta\zeta^2|\tilde{D}\partial_{k}\psi|^2+\frac{1}{4\delta}\big(|\tilde{
D}\psi|^2+|\partial_{k}\psi|^2\big)\Big),
 \end{split}
\end{equation}
where $0<\delta\leq1/2$ and $C=C(W,U,\zeta,\Sigma)$.

Next, for all $\varepsilon\in(0,\varepsilon_{U}')$ with
$\varepsilon_{U}'<\varepsilon_{U}$ sufficiently small, we obtain
\begin{equation}\label{eq:B_est}
 \begin{split}
 |B|&\leq C'\int_{U\times I}(|\tilde{D}\psi|+|f_{\varepsilon}|+|\psi|)|v|\\
 &\leq C'\int_{U\times
I}\Big(\alpha|v|^2+\frac{3}{4\alpha}\big(|\tilde{D}\psi|^2+|f_{\varepsilon}
|^2+|\psi|^2\big)\Big),
 \end{split}
\end{equation}
where $\alpha>0$, and $C'=C'(W,U,\zeta,\Sigma)$. We  also have
\begin{equation}\label{eq:v_est}
 \int_{U\times I}|v|^2\leq C''\left(\int_{U\times
I}|\partial_{k}\psi|^{2}+\int_{U\times I}\zeta^2|\partial_{k}^{2}\psi|^2\right)
\end{equation}
with $C''=C''(W,U,\zeta)$.

Now we will distinguish two cases:

$\bullet \boldsymbol{k=1,2,\ldots,d-1:}$
Using \eqref{eq:A_1'_est}, \eqref{eq:A_1''_est}, \eqref{eq:B_est}, and
\eqref{eq:v_est}, the  estimate
\begin{equation}\label{eq:reg_est}
A_{1}'-|A_{1}''|+\varepsilon^{-2}A_{2}\leq |B|
\end{equation}
yields
\begin{multline}\label{eq:H^2_main_est}
  (\gamma-2C\delta)\int_{U\times
I}\zeta^2|\tilde{D}\partial_{k}\psi|^2-C'C''\alpha\int_{U\times
I}\zeta^2|\partial_{k}^{2}\psi|^2\\
 +\varepsilon^{-2}\left(\int_{U\times
I}\zeta^2|\partial_{u}\partial_{k}\psi|^2-E_{1}\int_{U\times
I}\zeta^2|\partial_{k}\psi|^2\right)\\
 \leq \Big(\frac{C}{2\delta}+C'C''\alpha\Big)\int_{U\times
I}|\partial_{k}\psi|^2+\Big(\frac{C}{2\delta}+\frac{3C'}{4\alpha}\Big)\int_{
U\times I}|\tilde{D}\psi|^2\\
 +\frac{3C'}{4\alpha}\int_{U\times I}(|f_{\varepsilon}|^2+|\psi|^2).
\end{multline}
All the terms on the right hand side are of the order
$\mathcal{O}(\varepsilon^2)$. Therefore, with~$\delta$ and~$\alpha$ small enough
we have
\begin{align}
  &\int_{U\times I}\zeta^2|\tilde{D}\partial_{k}\psi|^2
  =\mathcal{O}(\varepsilon^2),\label{eq:mixed_der_conv_D}\\
  &\int_{U\times I}\zeta^2|\partial_{u}\partial_{k}\psi|^2
  =\mathcal{O}(\varepsilon^{2}).\label{eq:mixed_der_conv}
\end{align}

$\bullet \boldsymbol{k=d:}$
From \eqref{eq:mixed_der_conv} and the inequality
$$A_{1}'+\varepsilon^{-2}\int_{U\times I}\zeta^2 |\partial_{u}^{2}\psi|^{2}\leq
|B|+|A_{1}''|+\varepsilon^{-2}E_{1}\int_{U\times I}\zeta^2|\partial_{u}\psi|^2$$
we show in a similar manner as above that
\begin{equation}\label{eq:uu_der_conv}
\int_{U\times I}\zeta^2|\partial_{u}^{2}\psi|^2=\mathcal{O}(\varepsilon^2).
\end{equation}

\subsection{$\mathcal{H}^2$-convergence up to the boundary} 
In this subsection we will additionally assume that $\Sigma$ admits a uniformly
positive definite metric, and so estimate \eqref{eq:A_1'_est} remains valid up to the boundary.

Under \ref{eq:as3}, $\partial\Sigma$ is definitely $C^2$-smooth, which means
that for every boundary point $x_{\shortparallel,0}$ from the range of a local
boundary coordinate map (denote this  range by $R$) there exists a ball,
$B_{x_{\shortparallel,0}}(r)$, and $\alpha\in\{1,2,\ldots, d-1\}$ such that
$$R\cap B_{x_{\shortparallel,0}}(r)=\{(x_1,x_2,\ldots,x_{d-1})\in
B_{x_{\shortparallel,0}}(r)\,|\,
x_{\alpha}\geq\gamma(x_1,\ldots,\hat{x}_{\alpha},\ldots,x_{d-1})\},$$
where $\gamma\in C^2(\R^{d-2})$. This makes it possible to pass to  new local
coordinates with a straightened boundary; 
in these coordinates (denoted by the
same symbols) defined on some open half-ball of radius $\delta$, $U(\delta)$, 
we have $x_\alpha=0$ and $\psi=0$ on the straight part of $\partial
U(\delta)\times I$. Consequently,
$\partial_{j}\psi=\partial_{j}\partial_{l}\psi\equiv 0$ on
$\{x_{\shortparallel}\in\partial U(\delta)\,|\, x_\alpha=0\}\times I$, for
$j,l\neq\alpha$.

Now we may continue in the similar manner as when proving the convergence on the
``inner'' domains. More concretely, we set $U=U(\delta)$ and introduce a new
bump function that equals $1$ on a half-ball that has the common centre as
$U(\delta)$ but a smaller radius and it is completely contained in $U(\delta)$.
Beware that~\eqref{eq:per_partes} does not hold for $k=\alpha$ anymore.
Nevertheless it still holds true for $k\neq\alpha$, 
and so for  $k\in\{1,\ldots,\hat{\alpha},\ldots,d-1\}$,
\eqref{eq:mixed_der_conv_D} and \eqref{eq:mixed_der_conv} remain valid too.  If
we prove \eqref{eq:mixed_der_conv} for $k=\alpha$, we will also have
\eqref{eq:uu_der_conv}.
Moreover if we prove that
$$\int_{U(\delta)\times
I}\zeta^2|\partial^2_{\alpha}\psi|^2=\mathcal{O}(\varepsilon^2),$$
we shall obtain the local $\mathcal{H}^{2}$-convergence up to the boundary. 

But this may be done as follows. In the decomposition
\begin{equation*}
 \langle-\partial_{\mu}G^{\mu\nu}\partial_{\nu}\psi,
-\partial_\alpha\zeta^2\partial_\alpha\psi\rangle_{U(\delta)\times
I}=\tilde{A}_{1}'+\tilde{A}_{1}'',
\end{equation*}
with
\begin{align*}
 \tilde{A}_{1}':=&\langle
G^{\alpha\alpha}\partial^{2}_{\alpha}\psi,\zeta^2\partial^{2}_{\alpha}
\psi\rangle_{U(\delta)\times I},\\
 \begin{split}\tilde{A}_{1}'':=&\langle
G^{\alpha\alpha}_{,\alpha}\partial_{\alpha}\psi,\zeta^2\partial^{2}_{\alpha}
\psi\rangle_{U(\delta)\times I}+\langle
G^{\alpha\alpha}\partial^{2}_{\alpha}\psi,2\zeta\zeta_{,\alpha}\partial_{\alpha}
\psi\rangle_{U(\delta)\times I}\\
 &+\langle
G^{\alpha\alpha}_{,\alpha}\partial_{\alpha}\psi,2\zeta\zeta_{,\alpha}\partial_{
\alpha}\psi\rangle_{U(\delta)\times
I}+\sum_{\mu\neq\alpha\vee\nu\neq\alpha}\langle\partial_{\mu}
G^{\mu\nu}\partial_{\nu}\psi,\zeta^2
\partial^{2}_{\alpha}\psi\rangle_{U(\delta)\times I}\\
 &+\sum_{\mu\neq\alpha\vee\nu\neq\alpha}\langle\partial_{\mu}
G^{\mu\nu}\partial_{\nu}\psi,2\zeta\zeta_{,\alpha}\partial_{\alpha}\psi\rangle_{
U(\delta)\times I},
 \end{split}
\end{align*}
we may estimate $\tilde{A}'_{1}$ from below and $|\tilde{A}''_{1}|$ from above
in the same way as we have estimated $A'_{1}$ and $|A''_{1}|$, respectively.
Finally, the inequality
$$\tilde{A}_{1}'-|\tilde{A}_{1}''|+\varepsilon^{-2}A_{2}\leq |B|$$
leads to an equation that is similar to \eqref{eq:H^2_main_est} which in turn
yields the desired result.

\subsection{Interior $\mathcal{H}^k$-convergence and $C^0$-convergence}
In this section we show that under some stronger regularity assumptions on
$\Sigma$, for any $k\geq 3$,
\begin{equation*}
 \|\psi\|_{\mathcal{H}^k(W\times I)}=\mathcal{O}(\varepsilon),
\end{equation*}
whenever we have this result for $k-1$. Within this section $W$ and $U$ are the
same as in Section \ref{sec:H2_conv}. Let us assume that
\begin{equation}\tag*{$\langle A4\rangle$}\label{eq:H^k_assumptions}
\Sigma\in C^l\quad\text{with }l\in\N,\, l>2(k+1)+\frac{d}{2}, 
\end{equation}
which implies that all derivatives of $|g|^{\pm 1/2}$ up to the order $k-1$ are
bounded on any compact set, all derivatives of $G^{\mu\nu}$ up to the same order
are uniformly bounded in $\varepsilon$ (on some right neighbourhood of zero) on
any compact set, and
$$
 \sup_{U\times I}|D^\alpha V-D^\alpha V_{\mathrm{eff}}|
  =\mathcal{O}(\varepsilon) 
  \,, \qquad
 \sup_{U\times I}|D^\alpha a^{\mu\nu}|
 =\mathcal{O}(\varepsilon) \,,
$$
where $\alpha$ is a $d$-dimensional multiindex with $|\alpha|=k-1$. In fact, the
implications above hold true also under a weaker regularity assumption 
than~\ref{eq:H^k_assumptions}, but we need \ref{eq:H^k_assumptions} for a
sufficient regularity of~$\psi$ anyway
(cf.\ the definition of~$v$ below).

One may suggest proceeding
as in the proof of the $\mathcal{H}^2$-convergence but start with
$$v=(-1)^{|\alpha|}D^{\alpha}\zeta^2 D^{\alpha}\psi$$
instead of \eqref{eq:test_funct}. Here
$D^\alpha=\partial^{\alpha_{1}}_{1}\ldots\partial^{\alpha_{d}}_{d}$. However in
this way we obtain
\begin{equation}\label{eq:k_Sobolev_est}
 \|\partial_k D^\alpha \psi\|_{W\times I}=\mathcal{O}(\varepsilon)
\end{equation}
only if $\alpha_d=0$. This restriction is necessary for diminishing of boundary
terms during the integration by parts (cf.\ \eqref{eq:per_partes},
\eqref{eq:per_partes2}, and \eqref{eq:per_partes3}).
Also remark that in comparison to 
Section~\ref{sec:H2_conv}  
we have to treat the term denoted by~$B$ more carefully. 
At first we integrate by parts $(|\alpha|-1)$-times to get
\begin{align*}
  B=\ &\langle-(|g|^{-1/2})_{,\mu}G^{\mu\nu}|g|^{1/2}
  \partial_{\nu}\psi+f_{\varepsilon}+(\lambda_{n}-V)\psi,v\rangle_{U\times I}
  \\
  =\ &-\big\langle -D^{\tilde{\alpha}}(|g|^{-1/2}),_{\mu}G^{\mu\nu}|g|^{1/2}
  \partial_\nu\psi+D^{\tilde{\alpha}}f_\varepsilon
  \\
  & \qquad +D^{\tilde{\alpha}}
  [(\lambda_n-V)\psi], \,
  2\zeta\zeta,_\varrho D^\alpha \psi+\zeta^2 
  \partial_\varrho D^\alpha \psi\big\rangle_{U\times I}
  \,.
\end{align*}
Here $\varrho\in\{1,\ldots, d-1\}$ is such that $\alpha_\varrho\neq 0$ and  the
multiindex $\tilde{\alpha}$ is given by the decomposition
$\alpha=\tilde{\alpha}+(0,\ldots,0,1,0,\ldots 0)$, where, in the latter
multiindex, the only non-zero entry stands on the $\varrho$th position.
Next we find a bound for $|B|$ in a similar manner as in \eqref{eq:B_est} and
finally, in  \eqref{eq:reg_est}, we put on the left-hand-side only that term
from this bound that contains
$$\int_{U\times I}\zeta^2 |\partial_\rho D^\alpha\psi|^2.$$

Now, with \eqref{eq:k_Sobolev_est} in hand, let us focus on higher derivatives
that contain at least the second derivative with respect to~$u$. 
We differentiate~\eqref{eq:psi_diff} with respect to $x_\gamma$ and then rewrite
the result as
\begin{equation*}
-\partial_\gamma\partial_{u}^{2}\psi=\varepsilon^2\partial_\gamma
|g|^{-1/2}\partial_\mu |g|^{1/2}G^{\mu\nu}\partial_\nu \psi+E_1 \partial_\gamma
\psi+\varepsilon^2 \partial_\gamma f_\varepsilon +\varepsilon^2 \partial_\gamma
(\lambda_n-V) \psi.
\end{equation*}
Next we take $L^2(W\times I)$-norm of both sides to get
\begin{multline*}
 \|\partial_\gamma\partial_{u}^{2}\psi\|_{W\times I}\leq  \varepsilon^2
\|G^{\mu\nu}\|_\infty \|\partial_\gamma\partial_\mu\partial_\nu \psi\|_{W\times
I}+\varepsilon^2 C''' \|\psi\|_{\mathcal{H}^2(W\times I)}\\
 +E_1 \|\partial_\gamma \psi\|_{W\times I}+\varepsilon^2 (\mathrm{vol}(W\times
I))^{1/2} \sup_{W\times I}|\partial_{\gamma}f_\varepsilon|,
\end{multline*}
where $C'''=C'''(W,\Sigma,n)$.
Using \eqref{eq:k_Sobolev_est} and the above mentioned consequences of
\ref{eq:H^k_assumptions} we conclude that
\begin{equation*}
  \|\partial_\gamma\partial_{u}^{2}\psi\|_{W\times I}=\mathcal{O}(\varepsilon)
\end{equation*}
as $\varepsilon\to 0$. 
Differentiating \eqref{eq:psi_diff} with respect to $u$ we obtain, in a similar
manner,
\begin{equation*}
  \|\partial_{u}^{3}\psi\|_{W\times I}=\mathcal{O}(\varepsilon).
\end{equation*}
It is clear that one may use the same strategy 
to estimate also all higher derivatives by a successive differentiation.

\begin{theorem}\label{theo:Hk_conv}
Assume \ref{eq:eps_bound}, \ref{eq:res_conv_ass}, and
\begin{equation}\tag*{$\langle A5\rangle$}\label{eq:reg_ass_final}
\Sigma\in C^l\quad\text{with }l\in\N,\,l>\lfloor 3d/2\rfloor+4
\end{equation}
(there may be even an equality for $d$ odd).
 Then for any open and bounded subset $W$ such that its closure lies in the
range of a local coordinate map on~$\Sigma$, we have
\begin{equation*}
 \|\psi\|_{\mathcal{H}^k(W\times I)}=\mathcal{O}(\varepsilon)
\end{equation*}
with $k=\lfloor d/2\rfloor+1$.
Moreover,
\begin{equation*}
 \|\psi\|_{C^{0,\gamma}(\overline{W\times
I})}\equiv\|\psi\|_{C(\overline{W\times I})}+\sup_{x,y\in W\times I,x\neq
y}\left\{\frac{|\psi(x)-\psi(y)|}{|x-y|^{\gamma}}\right\}=\mathcal{O}
(\varepsilon),
\end{equation*}
where $0<\gamma<1$ for $d$ even and $\gamma=1/2$ for $d$ odd.
\end{theorem}

The second statement follows directly from 
the general Sobolev embedding 
$\mathcal{H}^k(W \times I) \hookrightarrow C^{0,\gamma}(\overline{W \times I})$,
cf.~\cite[Thm.~5.4]{adams}.

\subsection{Interior $C^{0,1}$-convergence in the transverse direction}
This subsection is strong\-ly inspired by a similar one in~\cite{FK4}. 
However, we had to somehow refine the supersolution used within the proof.

Let $B_{1},B_{2}$ be two 
$(d-1)$-dimensional
balls with the common centre 
$x_{\shortparallel,0}$ 
and radii $r_{1},r_{2}$, respectively, such that $r_{1}<r_{2}$ and $B_{2}$ lies
inside the range of a local coordinate map on $\Sigma$. Then 
$$h(x_{\shortparallel}):=\chi_{\overline{B_{2}}\setminus
B_{1}}(x_{\shortparallel})(|x_{\shortparallel}-x_{\shortparallel,0}|-r_{1})^3
\ \in \ C^{2}(\overline{B_{2}}).$$

Now, rewrite \eqref{eq:psi_diff} as
$$M\psi=F_{\varepsilon}$$
with
\begin{align*}
 &M :=
  -\varepsilon^2 |g|^{-1/2}\partial_{\mu}|g|^{1/2}G^{\mu\nu}\partial_{\nu}
  -\partial_{u}^{2} \,, \\
 &F_{\varepsilon} :=
  \varepsilon^2 (\lambda_{n}-V)\psi+\varepsilon^2 f_{\varepsilon}+E_{1}\psi \,.
\end{align*}
Let us introduce, on $B_{2}\times I$,
$$w(x_{\shortparallel},u) :=
\beta\left(\mathrm{e}^{-(1+h(x_{\shortparallel}))^{-1}}-\mathrm{e}^{-|u\pm
2|}\right),$$
where $\beta>0$ will be determined below. We have
\begin{equation*}
 Mw=\beta\big(-\varepsilon^2
|g|^{-1/2}\partial_{\mu}|g|^{1/2}G^{\mu\nu}\partial_{\nu}\mathrm{e}^{-(1+h(x_{
\shortparallel}))^{-1}}+\mathrm{e}^{-|u\pm 2|}\big)\geq
\frac{\beta}{2}\mathrm{e}^{-3}
\end{equation*}
for all $\varepsilon$ smaller than some sufficiently small $\varepsilon_{1}$.
Furthermore, 
\begin{equation}\label{eq:boundary_est}
w|_{\partial B_{2}\times
I}(x_{\shortparallel},u)\geq\beta\big(\mathrm{e}^{-(1+(r_{2}-r_{1})^3)^{-1}}
-\mathrm{e}^{-1}\big)=:\beta L
\end{equation}
and $w\geq 0$ on $\overline{B_{2}\times I}$.

If we set (under the assumption that the both norms are finite)
$$
  \beta:=
\max\left\{2\mathrm{e}^3\|F_{\varepsilon}\|_{C(\overline{B_{2}\times
I})},L^{-1}\|\psi\|_{C(\overline{\partial B_{2}\times I})} \right\},$$
then
\begin{align*}
 Mw&\geq M(\pm\psi) &&\text{in}\quad B_{2}\times I \,,\\
 w&\geq\pm\psi &&\text{on}\quad \partial(B_{2}\times I) \,,
\end{align*}
which implies
$$|\psi|\leq w\quad \text{on } \overline{B_{2}\times I}$$
in virtue of the maximum principle \cite[Chap.2, Theo. 6]{PW}. Namely for all
$(x_{\shortparallel},u)\in B_{1}\times I$,
\begin{equation*}
 |\psi(x_{\shortparallel},u)|\leq
w(x_{\shortparallel},u)=\beta\big(\mathrm{e}^{-1}-\mathrm{e}^{-|u\pm
2|}\big)=\beta\mathrm{e}^{-1}\big(1-\mathrm{e}^{\mp
u-1}\big)\leq\beta\mathrm{e}^{-1}|u\pm 1|.
\end{equation*}

Under the assumptions  of Theorem \ref{theo:Hk_conv},
$\beta=\mathcal{O}(\varepsilon)$, and we have
\begin{theorem}
Assume \ref{eq:eps_bound}, \ref{eq:res_conv_ass}, and \ref{eq:reg_ass_final}.
Let $B$ be an open ball that lies inside the range of a local coordinate map on
$\Sigma$. Then constants $\varepsilon_{2}>0$ and $K$ that depend on $\Sigma$ and
$B$ (and in particular on $\dist(\partial B,\partial\Sigma)$) exist such that
for all $\varepsilon<\varepsilon_{2}$,
\begin{equation}\label{eq:trans_der}
 \sup_{(x_{\shortparallel},u)\in B\times
I}\frac{|\psi(x_{\shortparallel},u)|}{|u\pm 1|}\leq K\varepsilon.
\end{equation}
\end{theorem}

\begin{remark}[$C^{0,1}$-convergence in the transverse direction up to the
boundary for $d=2,3$]\label{rem:C_conv} Once we have $\mathcal{H}^2$-convergence
up to the boundary, we can easily modify argument above to prove also
$C^{0,1}$-convergence in the transverse direction up to the boundary in the case
$d=2,3$. It is sufficient to change $B_{j}$ for some half-balls 
in the range of a local boundary coordinate map.
\eqref{eq:boundary_est} then holds only on the non-flat component of $\partial B_{2}$.
However on the remaining component of the
boundary, $\psi=0$, thus it is also bounded above by $w(\geq 0)$.
\end{remark}

\section{Local convergence of nodal sets}\label{Sec.nodal}
In this section we fix $N\in\N$ and consider only $\varepsilon<\varepsilon(N)$,
where $\varepsilon(N)$ is given in Remark \ref{rem:H_0_spec}.

\begin{lemma}\label{lem:cut}
Let $\mathcal{N}\equiv\mathcal{N}(\phi_n)$ be the nodal set of $\phi_{n}$ and
$B$ be an open ball inside the range of a local coordinate map on $\Sigma$ which
is split by $\mathcal{N}$ 
to exactly two (open) non-empty components
$D_{1}$ and $D_{2}$ such that $\phi_{n}(x_{\shortparallel})\neq 0$ 
for all $x_{\shortparallel}\in D_{i}$. Then $\phi_{n}$ is positive on $D_{1}$
and negative on $D_{2}$ (or vice versa, but we will set the notation in the
manner that the former is always true).
 Moreover, if there exists a ball in $D_{2}$ that touches $\mathcal{N}$ at
$x_{\shortparallel}\in\mathcal{N}\cap B$  we have
 $$\frac{\partial\phi_{n}}{\partial
n_{x_{\shortparallel}}}(x_{\shortparallel})>0,$$
 where $n_{x_{\shortparallel}}$ is the normal vector to $\mathcal{N}$ at
$x_{\shortparallel}$ pointing out from $D_{2}$.
\end{lemma}
\begin{proof}
Assume that $\phi_{n}$ does not change its sign on $B$, e.g. $\phi_{n}\leq 0$.
Thus $\phi_{n}$ attains its maximum $(=0)$ at all points of $\mathcal{N}\cap B$.
We have
$$\Delta_{g}\phi_{n}-(V_{\mathrm{eff}}-\sigma_n)_{+}\phi_{n}=-(V_{\mathrm{eff}}
-\sigma_n)_{-}\phi_{n}\geq 0$$
on $B$. 
The maximum principle \cite[Chap.~2, Theo.~6]{PW} now says that if $\phi_{n}$
attains a non-negative maximum $M$ at an interior point of~$B$ 
then $\phi_{n}\equiv M=0$ on~$B$.  
But this implies that $\phi_n\equiv 0$ on $\Sigma$ which is not possible,
because $\phi_n$ is supposed to be an eigenfunction.

Theorem~8 of \cite[Chap.~2]{PW} applied on~$D_{2}$ 
then implies the second assertion of the lemma.
\end{proof}

Recall that in general the nodal set for a solution $u$ of an elliptic equation
in $\Omega\subset\R^m$ decomposes into a disjoint union of a $(m-1)$-dimensional
manifold (that contains all the points where $|\nabla u|> 0$) and a closed
countably $(m-2)$-rectifiable subset (that contains all the points where
$|\nabla u|=0$), cf.~\cite{Caffarelli-Friedman_1985} and
\cite{Hardt-Simon_1989}. The latter component is called \emph{the singular set}.
Since the eigenfunction $\phi_n$ is at least $C^3$-smooth under our regularity
assumptions, the former component of $\mathcal{N}(\phi_n)$ is surely a
$C^3$-smooth manifold, and so the assumptions of  Lemma \ref{lem:cut} are
fulfilled away from the singular set.

In particular, for any $z_{\shortparallel}\in \mathcal{N}(\phi_n)$ such that
$|\nabla{\phi_n}(z_{\shortparallel})|>0$, there exists a ball $B$ such that
$z_{\shortparallel}\in B$, $|\nabla{\phi_n}(x_{\shortparallel})|>0$ for all
$x_{\shortparallel}\in B$ , and Lemma \ref{lem:cut} holds true with this $B$. 
We will denote by $\tau_\delta$ a closed tubular neighbourhood 
about a piece of~$\mathcal{N}(\phi_n)$ of radius~$\delta$ contained in $B$. Let
$D_{1}$ and $n_{x_{\shortparallel}}$ be as in Lemma \ref{lem:cut}.
Take $x_{\shortparallel}\in\tau_\delta\cap D_{1}$, then there exists
$x_{\shortparallel}'\in\tau_\delta\cap\mathcal{N}(\phi_n)$ with the property
$x_{\shortparallel}-x_{\shortparallel}'=|x_{\shortparallel}-x_{\shortparallel}
'|n_{x_{\shortparallel}'}$
($x_{\shortparallel}'$ is not given uniquely but we always have
$|x_{\shortparallel}-x_{\shortparallel}'|\geq
\dist(x_{\shortparallel},\mathcal{N}(\phi_n))$).
Using the Taylor series expansion, we have
\begin{equation*}
 \phi_{n}(x_{\shortparallel})=\phi_{n}(x_{\shortparallel})-\phi_{n}(x_{
\shortparallel}')=\frac{\partial\phi_{n}}{\partial
n_{x_{\shortparallel}'}}|x_{\shortparallel}-x_{\shortparallel}'|+\mathcal{O}(|x_
{\shortparallel}-x_{\shortparallel}'|^2).
\end{equation*}
Using the explicit form of the remainder in the expansion and the regularity of
$\phi_{n}$, we infer that for $\delta$ small enough there exists a positive
constant $A$ such that
\begin{equation}\label{eq:phi_N_bound}
|\phi_{n}(x_{\shortparallel})|\geq
A\,\dist(x_{\shortparallel},\mathcal{N}(\phi_n))\quad\text{for all
}x_{\shortparallel}\in\tau_\delta. 
\end{equation}

Since $\psi_{n}^{0}=\phi_{n}\otimes\chi_{1}$,
where~$\chi_{1}$ is given by~\eqref{chi1},
we have
\begin{equation*}
|\psi_{n}^{0}(x_{\shortparallel},u)|\geq
A\,\dist(x_{\shortparallel},\mathcal{N}(\phi_{n}))\dist(u,\partial I),\quad
\forall (x_{\shortparallel},u)\in\tau_\delta\times I.
\end{equation*}
Putting this result together with \eqref{eq:trans_der} we have
\begin{equation*}
 \frac{\psi_{n}\,\sgn\phi_{n}}{\dist(u,\partial I)}\geq
A\,\dist(x_{\shortparallel},\mathcal{N}(\phi_{n}))-K\varepsilon,
\end{equation*}
which implies that, in $\tau_\delta\times I$,
\begin{equation*}
\sgn\psi_{n}(x_{\shortparallel},u)=\sgn\phi_{n}(x_{\shortparallel}) 
\end{equation*}
whenever $\dist(x_{\shortparallel},\mathcal{N}(\phi_{n}))>KA^{-1}\varepsilon$.
Consequently, if $\varepsilon<AK^{-1}\delta$ then, for any $u\in I$, $x \mapsto
\psi_n(x,u)$ changes its sign on $\tau_\delta$.
We may extend this result along a compact piece of~$\mathcal{N}(\phi_n)$ to
obtain
\begin{theorem}\label{theo:sign_change}
 Assume \ref{eq:eps_bound}, \ref{eq:res_conv_ass}, and \ref{eq:reg_ass_final}.
Let $\tau_\delta$ be a tubular neighbourhood of radius~$\delta$ (with $\delta$
so small that \eqref{eq:phi_N_bound} holds true in $\tau_\delta$ with some
positive constant $A$) about a compact piece of~$\mathcal{N}(\phi_n)$ that lies
inside the range of a local coordinate map on $\Sigma$ and that does not
intersect the singular set of~$\mathcal{N}(\phi_n)$. Then there exists a
positive constant $\tilde{A}$ such that, for any $u\in I$, 
$x \mapsto \psi_n(x,u)$ changes its sign on $\tau_\delta$, whenever
$\varepsilon<\tilde{A}\delta$.
\end{theorem}

This immediately implies Theorem \ref{Thm.Conj}, because $\mathcal{N}(\phi_n)$
can not equal to its singular set.

\section{Some more results in special cases}\label{Sec.special}
By the theorem of Courant's, $\phi_n$ has at most $n$ nodal domains. Let us
denote them $N_i,\, i=1,\ldots,m$, where $m\leq n$. 
Consider the following special cases.

\subsection{$\boldsymbol{\Sigma}$ compact}
%
For $1>\delta>0$, let us define 
$$
  N_i(\delta):=\{x\in N_i \,|\, \dist{(x,\partial
  N_i})\geq\delta\}
$$ 
and
$m_i(\delta):=\inf_{N_i(\delta)}|\phi_n|\cos{((1-\delta)\pi/2)}$. 
Clearly,
$m_i(\delta)>0$. Using Theorem~\ref{theo:Hk_conv} we infer that there exist
positive constants $\varepsilon_\mathrm{lim}$ and $\tilde{K}$ such that, for all
$\varepsilon\leq\varepsilon_\mathrm{lim}$, we have 
\begin{equation}\label{eq:unif_conv}
\|\psi\|_{C(\overline{\cup_{i=1}^{m}N_i(\delta)\times
I})}\leq\tilde{K}\varepsilon.
\end{equation}
If $d=2,3$ we have even stronger result due to Remark \ref{rem:C_conv},
\begin{equation*}
\|\psi\|_{C(\overline{\Sigma\times I})}\leq\tilde{K}\varepsilon,\quad
\forall\varepsilon\leq\varepsilon_\mathrm{lim},
\end{equation*}
with $\tilde{K}$ and $\varepsilon_\mathrm{lim}$ being $\delta$-independent (on
the contrary to the generic case).
Using \eqref{eq:unif_conv} we obtain
$$\psi_n \sgn{\phi_n}\geq|\phi_n\chi_1|-\tilde{K}\varepsilon\geq
m_i(\delta)-\tilde{K}\varepsilon$$
in $N_i(\delta)\times (-1+\delta,1-\delta)$. We conclude that for 
\begin{equation}\label{eq:dom_conv_bound}
\varepsilon<\tilde{K}^{-1}\min{\{m_i(\delta),\ i=1,\ldots, n\}},
\end{equation}
$\sgn{\psi_n}(x_{\shortparallel},u)=\sgn{\phi_n}(x_{\shortparallel})$ in
$\cup_{i=1}^{m}N_i(\delta)\times (-1+\delta,1-\delta)$. This means that, for all sufficiently small 
$\varepsilon$, the number of nodal domains of $\psi_n$ cannot be smaller than $m$ and that $N_i(\delta)\times
(-1+\delta,1-\delta)$ is contained in a nodal domain of $\psi_n$, which  will be henceforth called the 
$i$th nodal domain of $\psi_n$.
Indeed, if the number of nodal domains of $\psi_n$ was smaller that $m$, it
would contradict Theorem \ref{theo:sign_change}. We conclude that the Lebesgue measure of the symmetric
difference of the $i$th nodal domain of $\psi_n$ and $N_i\times I$ tends to zero
as $\varepsilon\to 0$, and that $\mathcal{N}(\psi_n)$ lies in 
$L(\delta):=\Sigma\times I\setminus\cup_{i=1}^{m}N_i(\delta)\times(-1+\delta,1-\delta)$
for all $\varepsilon\leq\varepsilon_\mathrm{lim}$.

\subsection{$\boldsymbol{n=2, \Sigma$ compact and closed,
$\boldsymbol{\mathcal{N}(\phi_n)}}$ has empty singular set}
%
For any second eigenfunction there are exactly two nodal domains. Since $\Sigma$
is compact and the singular set of $\mathcal{N}(\phi_2)$ is empty, we can use
Theorem~\ref{theo:sign_change} along the whole nodal set $\mathcal{N}(\phi_2)$.
Then for all $\varepsilon<\min{\{\tilde{A}_{i}\delta_i\}}$,
$\mathcal{N}(\psi_2)$ is localised in the Cartesian product of the tubular neighbourhood of radius
$\max{\{\delta_i\}}$ about $\mathcal{N}(\phi_2)$ times $I$. So the nodal set of $\psi_2$
converges in the Hausdorff sense to the set $\mathcal{N}(\phi_2)\times I$, as
$\varepsilon\to 0$. Similarly, both the nodal domains of $\psi_2$ converge in
the Hausdorff sense to the Cartesian product of the nodal domains of $\phi_2$
times $I$.

In fact, one can locate
the nodal set and the nodal domains of~$\psi_n$ in a 
much more precise way
if $d=2$ (see~\cite{FK4}). 
The reason is that in that case we have much more explicit knowledge about
$\phi_n$, since it is just an eigenfunction of some Sturm-Liouville operator. 

\subsection*{Acknowledgement}
One of the authors (M.T.) would like thank to the Institute Mittag-Leffler
(Djursholm, Sweden), where this work was started, 
for its kind hospitality. The work has been partially supported
by the project RVO61389005
and the grants No.\ 13-11058S and 14-06818S
of the Czech Science Foundation (GA\v{C}R).

%
\providecommand{\bysame}{\leavevmode\hbox to3em{\hrulefill}\thinspace}
\providecommand{\MR}{\relax\ifhmode\unskip\space\fi MR }
\providecommand{\MRhref}[2]{%
  \href{http://www.ams.org/mathscinet-getitem?mr=#1}{#2}
}
\providecommand{\href}[2]{#2}

\end{document}